\documentclass[runningheads]{llncs}

\usepackage{amsmath}
\usepackage{amsfonts}
\usepackage{amssymb}
\usepackage{mathrsfs}
\usepackage{stmaryrd}

\newcommand{\Gr}{Gr\"obner }

\newcommand{\cI}{\cal{I}}

\newcommand{\Q}{\mathbb{Q}}
\newcommand{\Z}{\mathbb{Z}}
\newcommand{\cR}{\cal{R}}
\newcommand{\N}{\mathbb{N}}

\usepackage[utf8x]{inputenc}

\usepackage{graphicx}
\usepackage{ucs}
\usepackage[utf8x]{inputenc}

\usepackage{subcaption}
\usepackage{tikz}
\usepackage[outline]{contour}

\usepackage[T1]{fontenc}
\usepackage{authblk}
\usepackage{booktabs}
\usepackage{hyperref}
\hypersetup{
    colorlinks,
    hidelinks,
    citecolor=blue,
    filecolor=black,
    linkcolor=red,
    urlcolor=green,
    plainpages=false,
    pdfpagelabels=true
}

\usepackage{url}
\newcommand{\keywords}[1]{\par\addvspace\baselineskip
\noindent\keywordname\enspace\ignorespaces#1}

\begin{document}

\mainmatter

\title{A Strongly Consistent Finite Difference Scheme for Steady Stokes Flow and its Modified Equations }
\titlerunning{A Strongly Consistent Finite Difference Scheme for Steady Stokes Flow}

\author{Yury~A.~Blinkov\inst{1}
\and Vladimir P.~Gerdt\inst{2,3} \and Dmitry~A.~Lyakhov\inst{4} \and Dominik L.~Michels~\inst{4}}
\institute{
Saratov State University, Saratov, 413100, Russian Federation\\
\email{BlinkovUA@info.sgu.ru}\\
\and
Joint Institute for Nuclear Research, Dubna, 141980, Russian Federation\\
\and
Peoples' Friendship University of Russia, Moscow, 117198, Russian Federation\\
\email{Gerdt@jinr.ru}\\
\and
King Abdullah University of Science and Technology, Thuwal, 23955-6900, Kingdom of Saudi Arabia\\
\email{\{Dmitry.Lyakhov,Dominik.Michels\}@kaust.edu.sa}\\
}
\authorrunning{Yu.~A.~Blinkov et al.}
\toctitle{Lecture Notes in Computer Science}
\tocauthor{Yu.~A.~Blinkov et al.}


\maketitle

\begin{abstract}
We construct and analyze a strongly consistent second-order finite difference scheme for the steady two-dimensional Stokes flow.
The pressure Poisson equation is explicitly incorporated into the scheme. Our approach suggested by the first two authors
is based on a combination of the finite volume method, difference elimination, and numerical integration. We make use of the
techniques of the differential and difference Janet/\Gr bases. In order to prove strong consistency of the generated scheme
we correlate the differential ideal generated by the polynomials in the Stokes equations with the difference ideal generated
by the polynomials in the constructed difference scheme. Additionally, we compute the modified differential system of the
obtained scheme and analyze the scheme's accuracy and strong consistency by considering this system. An evaluation of our scheme against the established marker-and-cell method is carried out.

\end{abstract}

\section{Introduction}
\label{sec:1}
In this paper, we consider the two-dimensional flow of an incompressible fluid described by the following system of partial differential equations (PDEs):
\begin{equation}
\label{Stokes}
\left\lbrace
\begin{array}{rl}
F^{(1)}:=& u_x+v_y=0\,,\\[4pt]
F^{(2)}:=& p_x - \frac{1}{\mathrm{Re}}\Delta\,u - f^{(1)}=0\,,\\[4pt]
F^{(3)}:=& p_y - \frac{1}{\mathrm{Re}}\Delta\,v-f^{(2)}=0\,.
\end{array}
\right.
\end{equation}
Here the velocities $u$ and $v$, the pressure $p$, and the external forces $f^{(1)}$ and $f^{(2)}$ are functions in $x$ and $y$; $\mathrm{Re}$ is the Reynolds number and $\Delta:=\partial_{xx}+\partial_{yy}$ is the Laplace operator.

A flow that is governed by these equations is denoted in the literature as a Stokes flow or a creeping flow. Correspondingly, the PDE system~\eqref{Stokes} is called a Stokes system.  It approximates the Navier--Stokes system for a two-dimensional incompressible steady flow when ${\mathrm{Re}} \ll 1$. The last condition makes the nonlinear inertia terms in the Navier--Stokes system much smaller then the viscous forces (cf. \cite{Milne-Thompson'68}, Sect.~22$\cdot$11), and neglecting of the nonlinear terms results in~Eqs.~\eqref{Stokes}. The fundamental mathematical theory of the Stokes flow is e.g.~presented in~\cite{KohrPop'04}.

Our first aim is to construct, for a uniform and orthogonal grid, a finite difference scheme for the governing system~\eqref{Stokes} which contains a discrete version of the pressure Poisson equation and whose algebraic properties are strongly consistent (or s-consistent, for brevity)~\cite{GR'10,G'12} with those of~Eqs.~\eqref{Stokes}. For this purpose, we use the approach proposed in~\cite{GBM'06} based on a combination of the finite volume method, numerical integration, and difference elimination. For the generated scheme we apply the algorithmic criterion to verify its s-consistency. The last criterion was designed in~\cite{GR'10} for linear PDE systems and then generalized in~\cite{G'12} to polynomially nonlinear systems. The computational experiments done in papers~\cite{ABGLS'13,ABGLS'17} with the Navier--Stokes equations demonstrated a substantial superiority in numerical behavior of s-consistent schemes over s-inconsistent ones.

The linearity of Eqs.~\eqref{Stokes} not only makes the construction and analysis of its numerical solutions much easier than in the case of the Navier--Stokes equations, but also admits a fully algorithmic generation of difference schemes for Eqs.~\eqref{Stokes} and their s-consistency verification. To perform related computations we use two Maple packages implementing the involutive algorithm (cf.~\cite{G'05}) for the computation of Janet and \Gr bases: the package {\sc Janet}~\cite{Maple-Janet'03}  for linear  differential systems and the  package {\sc LDA}~\cite{GR'12} (\underline{L}inear \underline{D}ifference \underline{A}lgebra) for linear difference systems.

Our second aim is to compute a modified differential system of the constructed difference scheme, i.e.,~modified Stokes flow, and to analyse the accuracy and consistency of the scheme via this differential system. Nowadays the method of modified equations suggested in~\cite{Shokin'83} is widely used~(see~\cite{GV'96}, Ch.~8 and \cite{Moin'10}, Sect.~5.5) in studying difference schemes.  The method provides a natural and unified platform to study such basic properties of the scheme as order of approximation, consistency, stability, convergence, dissipativity, dispersion, and invariance. However, as far as we know, the methods for the computation of modified equations have not been extended yet to non-evolutionary PDE systems. We show how the extension can be done for our scheme by applying the technique of differential Janet/\Gr bases.

The present paper is organized as follows. In Section~\ref{sec:scheme}, we generate for~Eqs.~\eqref{Stokes} a difference scheme by applying the approach of paper~\cite{GBM'06}. In Section~\ref{sec:consistency}, we show that our scheme is s-consistent and demonstrate s-inconsistency of another scheme obtained by a tempting compactification of our scheme. The computation of a modified Stokes system for our s-consistent scheme is described in Section~\ref{ModifiedEquation}. Here, we also show by the example of the s-inconsistent scheme of Section~\ref{sec:consistency} how the modified Stokes system detects the s-inconsistency.  Finally, a numerical benchmark against the marker-and-cell method is presented in Section~\ref{sec:Numerics} and some concluding remarks are given in Section~\ref{Conclusion}.

\section{Difference Scheme Generation for Stokes Flow}
\label{sec:scheme}
We consider the orthogonal and uniform solution grid with the grid spacing $h$ and apply the approach of paper~\cite{GBM'06} to generate a difference scheme for Eqs.~\eqref{Stokes}.\\

{\bf Step 1. Completion to involution} (we refer to~\cite{Seiler'10} and to the references therein for the theory of involution). We select the lexicographic  POT ({\underline{P}osition \underline{O}ver \underline{T}erm)~\cite{AdamsLoustaunau'94} ranking with
\begin{equation}\label{ranking1}
x\succ y\,,\quad u\succ v\succ p\succ f^{(1)}\succ f^{(2)}\,.
\end{equation}
Then the package {\sc Janet}~\cite{Maple-Janet'03} outputs the following {\em Janet involutive form} of Eqs.~\eqref{Stokes} which is the {\em minimal reduced differential \Gr basis form}:
\begin{equation}\label{InvSys}
\left\lbrace
\begin{array}{rl}
F^{(1)}:= &\ \underline{u_x}+v_y=0\,,\\[4pt]
F^{(2)}:= &\ p_x-\frac{1}{\mathrm{Re}}\left(\underline{u_{yy}}-v_{xy}\right)
 - f^{(1)}=0\,, \\[4pt]
F^{(3)}:= &\ p_y -\frac{1}{\mathrm{Re}}\left(\underline{v_{xx}} + v_{yy}\right) - f^{(2)}=0\,, \\[4pt]
F^{(4)}:= &\  \underline{p_{xx}} + p_{yy} - f^{(1)}_x - f^{(2)}_y=0\,.
\end{array}
\right.
\end{equation}
We underlined the {\em leaders}, i.e.,~the highest ranking partial derivatives occurring in Eqs.~\eqref{InvSys}.
$F^4$ is the {\em pressure Poisson equation} which, being the integrability condition for system~\eqref{Stokes}, is expressed in terms of its left-hand sides as
\begin{equation}\label{IntCon}
F^{(4)}:=F^{(2)}_x + F^{(3)}_y + \frac{1}{\mathrm{Re}}\left( F^{(1)}_{xx} + F^{(1)}_{yy}\right)  = p_{xx} + p_{yy} - f^{(1)}_x - f^{(2)}_y\,.
\end{equation}

\begin{remark}\label{Rem1} The differential polynomial $F^{(2)}$ in Eqs.~\eqref{InvSys} is $F^{(2)}$ in Eqs.~\eqref{Stokes} reduced modulo the continuity equation $F^{(1)}$.
\end{remark}

{\bf Step 2. Conversion into the integral form}. We choose the following integration contour $\Gamma$ as a ``control volume''
\begin{figure}[h]
    \centering
    \includegraphics[scale=0.8]{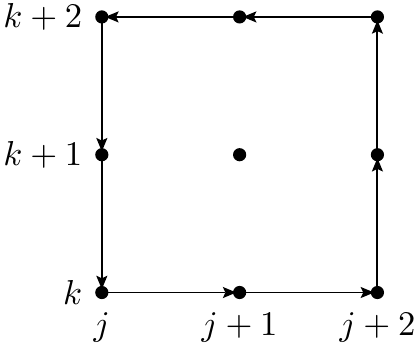}
    \caption{Integration contour $\Gamma$ (stencil $3\times 3$).}
    \label{IntegrationContour}
\end{figure}
and rewrite equations $F^{(1)},F^{(2)}$, and $F^{(3)}$ into the equivalent {\em integral form}
\begin{equation}\label{IntForm}
\left\lbrace
\begin{array}{rl}
  \begin{split}
& \oint _{\Gamma} \!-v \,dx + u \,dy = 0\,,\\[4pt]
& \oint _{\Gamma } \!\frac{1}{\mathrm{Re}}u_y \,dx + \left( p - \frac{1}{\mathrm{Re}}u_x\right)  \,dy -
\iint _{\Omega } f^{(1)}\! \,d x\,dy=0 \,, \\[4pt]
& \oint _{\Gamma} \! -\left(p - \frac{1}{ \mathrm{Re}}v_y\right)  \,dx - \frac{1}{\mathrm{Re}}v_x \,dy -
\iint _{\Omega } f^{(2)}\! \,d x\,dy=0 \,, 
  \end{split}
\end{array}
\right.
\end{equation}
where $\Omega$ is the internal area of the contour $\Gamma$.

It should be noted that we use in Eqs.~\eqref{IntForm} the original form of $F^{(2)}$ given in Eqs.~\eqref{Stokes} (see Remark~\ref{Rem1}) since we want to preserve at the discrete level the symmetry of system~\eqref{Stokes} under the swap transformation
\begin{equation}\label{SymTrans}
  \{x,u,f^{(1)}\} \longleftrightarrow \{y,v,f^{(2)}\}\,.
\end{equation}

{\bf Step 3. Addition of integral relations for derivatives.} We add to system~\eqref{IntForm} the {\em exact integral relations} between the partial derivatives of velocities and the velocities themselves:
\begin{equation}\label{relations}
\left\lbrace
\begin{array}{l}
\int \limits_{x_j}^{x_{j+1}} \! \! \! u_x dx = u(x_{j+1}, y) -
u(x_{j}, y)\,,
\quad
 \int \limits_{y_k}^{y_{k+1}} \! \! \! u_y dy = u(x, y_{k+1}) -
u(x, y_{k})\,, \\[0.6cm]
 \int \limits_{x_j}^{x_{j+1}} \! \! \! v_x dx = v(x_{j+1}, y) -
v(x_{j}, y)\,,
\quad
 \int \limits_{y_k}^{y_{k+1}} \! \! \! v_y dy = v(x, y_{k+1}) -
v(x, y_{k})\,.
\end{array}
\right.
\end{equation}

{\bf Step 4. Numerical evaluation of integrals.} We apply the midpoint rule for the contour integration in Eqs.~\eqref{IntForm}, the trapezoidal rule for the integrals \eqref{relations} and approximate the double integrals as
$$f^{(1,2)}_{i+1,k+1}4h^2\,,$$
where $h$ is the step of a square grid in the $(x,y)$ plane.

As a result, we obtain the difference equations for the grid functions
\[
   u_{j,\,k}\approx u(jh,kh)\,,\ v_{j,\,k}\approx v(jh,kh)\,,\ p_{j,\,k}\approx p(jh,kh)\,,\ f^{(1,2)}_{j,\,k}\approx f^{(1,2)}(jh,\,kh)
\]
approximating functions $u(x,y),v(x,y),p(x,y),f^{(1)}(x,y),f^{(2)}(x,y)$, and the grid functions approximating partial derivatives
\[
\left\lbrace
\begin{array}{l}
   {u_x}_{j,\,k}\approx u_x(jh,kh)\,,\quad {u_y}_{j,\,k}\approx u_y(jh,kh)\,,\\[10pt]
   {v_x}_{j,\,k}\approx v_x(jh,kh)\,,\,\,\quad {v_y}_{j,\,k}\approx v_y(jh,kh)\,,
\end{array}
\right.
\]
where $j,k\in \Z$:

\begin{equation}\label{DS}
\left\lbrace
\begin{split}
& \left(u_{j+2,\,k+1} -u_{j,\,k+1}\right) 2h + \left( v_{j+1,\,k+2}-v_{j+1,\,k}\right) 2h =0\, ,\\[4pt]
& \frac{1}{\mathrm{Re}}\left({u_y}_{j+1,\,k} -
{u_y}_{j+1,\,k+2}\right)2h + \left(p_{j+2,\,k+1}-\frac{1}{\mathrm{Re}}{u_x}_{j+2,\,k+1}\right)2h{}\\[4pt]
& \hfill {} -
\left(p_{j,\,k+1}-\frac{1}{\mathrm{Re}}{u_x}_{j,\,k+1}\right)2h
-4 f^{(1)}_{j+1,\,k+1}h^2=0\,, \\[4pt]
& -\left(\left(p_{j+1,\,k}-\frac{1}{\mathrm{Re}}{v_y}_{j+1,\,k}\right)-
\left(p_{j+1,\,k+2}-\frac{1}{\mathrm{Re}}{v_y}_{j+1,\,k+2}\right)\right)2h {}\\[4pt]
& \hfill + \left(-\frac{1}{\mathrm{Re}}{v_x}_{j+2,\,k+1} +
\frac{1}{\mathrm{Re}}{v_x}_{j,\,k+1}\right)2h\
-4 f^{(2)}_{j+1,\,k+1}h^2=0 \,,\\[4pt]
& \dfrac{{u_x}_{j+1,\,k} + {u_x}_{j,\,k}}{2}h - u_{j+1,\,k}+u_{j,\,k}=0\,,\\[4pt]
& \dfrac{{v_x}_{j+1,\,k} + {v_x}_{j,\,k}}{2}h - v_{j+1,\,k}+v_{j,\,k}=0\,,\\[4pt]
& \dfrac{{u_y}_{j,\,k+1} + {u_y}_{j,\,k}}{2}h - u_{j,\,k+1}+u_{j,\,k}=0\,,\\[4pt]
& \dfrac{{v_y}_{j,\,k+1} + {v_y}_{j,\,k}}{2}h - v_{j,\,k+1}+v_{j,\,k}=0\,.
\end{split}
\right.
\end{equation}

{\bf Step 5. Difference elimination of derivatives.} To eliminate the grid functions $u_x,u_y,v_x,v_y$ for the partial derivatives of the velocities, we construct a difference Janet/\Gr basis form of the set of linear difference polynomials in left-hand sides of~Eqs.~\eqref{DS} with the Maple package {\sc LDA}~\cite{GR'10} for the POT lexicographic ranking which is the difference analogue of the differential ranking used on Step~1:
\begin{equation}\label{ranking}
j\succ k,\quad u\succ v\succ p\succ f^{(1)}\succ f^{(2)}\,.
\end{equation}

The output of the {\sc LDA} includes four difference polynomials not containing the grid functions $u_x,u_y,v_x,v_y$. These polynomials comprise a difference scheme. Being interreduced, this scheme does not reveal a desirable discrete analogue of symmetry under the transformation~\eqref{SymTrans}. Because of this reason, we prefer the following redundant but symmetric form of the scheme:
\begin{equation}\label{J-Basis}
\left\lbrace
\begin{array}{rl}
\tilde{F}^{(1)}:=& \dfrac{u_{j+2,\,k+1} - u_{j,\,k+1}}{2h} + \dfrac{v_{j+1,\,k+2}-v_{j+1,\,k}}{2h}=0\, ,\\[8pt]
\tilde{F}^{(2)}:=& \dfrac{p_{j+2,\,k+1}-p_{j,\,k+1}}{2h} -
\dfrac{1}{\mathrm{Re}}\Delta_1\left(u_{j,k}\right)- f^{(1)}_{j+1,\,k+1}=0\,,\\[8pt]
\tilde{F}^{(3)}:=& \dfrac{p_{j+1,\,k+2}-p_{j+1,\,k}}{2h} -
\dfrac{1}{\mathrm{Re}}\Delta_1\left(v_{j,k}\right) - f^{(2)}_{j+1,\,k+1}=0\,,\\[8pt]
\tilde{F}^{(4)}:=& \Delta_2\left(p_{j,k}\right)
- \dfrac{f^{(1)}_{j+3,\,k+2} - f^{(1)}_{j+1,\,k+2}}{2h} - \dfrac{f^{(2)}_{j+2,\,k+3}-f^{(2)}_{j+2,\,k+1}}{2h}=0\,,
\end{array}
\right.
\end{equation}
where $\Delta_1$ and $\Delta_2$ are discrete versions of the Laplace operator acting on a grid function $g_{j,\,k}$ as
\begin{eqnarray}
 &&\Delta_1\left( g_{j,\,k}\right):=\dfrac{g_{j+2,\,k+1} + g_{j+1,\,k+2} -4 g_{j+1,\,k+1} + g_{j+1,\,k} + g_{j,\,k+1}}{h^2}\,,\\
  &&\Delta_2\left( g_{j,\,k}\right):=\dfrac{g_{j+4,\,k+2} + g_{j+2,\,k+4} -4 g_{j+2,\,k+2} + g_{j+2,\,k} + g_{j,\,k+2}}{4h^2}\,.
\end{eqnarray}

\begin{remark}\label{Rem2}
The difference equation $\tilde{F}^{(4)}$ of the system~\eqref{J-Basis} can also be obtained (cf.~\cite{GB'09}) from the integral form of $F^4$ in Eqs.~\eqref{InvSys}--\eqref{IntCon} with the contour illustrated in Fig.~\ref{IntegrationContour} by using the midpoint rule for the contour integration of the $p_x$ and $p_y$ as well as for evaluation of the additional integrals
\begin{equation}\label{p-relations}
\int \limits_{x_j}^{x_{j+2}} \! \! \! p_x dx = p(x_{j+2}, y) -
p(x_{j}, y)\,,
\quad
 \int \limits_{y_k}^{y_{k+2}} \! \! \! p_y dy = p(x, y_{k+2}) -
p(x, y_{k})\,,
\end{equation}
and the trapezoidal rule for the contour integration of $f^{(1)}$ and $f^{(2)}$.
\end{remark}

The difference polynomials~\eqref{J-Basis} approximate those in Eqs.~\eqref{InvSys}, and the correspondence between differential and difference Janet/\Gr bases is a consequence of our choice of the differential~\eqref{ranking1} and difference~\eqref{ranking} rankings.

\section{Consistency Analysis}
\label{sec:consistency}

Let ${\cR}={\Q}(\mathrm{Re},h)[u,v,p,f^{(1)},f^{(2)}]$ be the {\em ring of differential polynomials} over the field of rational functions in $\mathrm{Re}$ and $h$. We consider the functions describing the Stokes flow~\eqref{Stokes} as {\em differential indeterminates} and their grid approximations as {\em difference indeterminates}. Respectively, we denote by ${\tilde{\cR}}$ the {\em difference polynomial ring} whose elements are polynomials in the grid functions with the right-shift operators $\sigma_1$ and $\sigma_2$ acting as translations, for example,
\begin{equation}\label{shifts}
    \sigma_1\circ u_{j,\,k}=u_{j+1,\,k}\,,\quad \sigma_2\circ u_{j,\,k}=u_{j,\,k+1}\,.
\end{equation}
We denote by ${\cI}:=\langle F^{(1)},F^{(2)},F^{(3)}\rangle\subset {\cR}$ the {\em differential ideal} generated by the set of left-hand sides in~\eqref{Stokes} and by $\tilde{\cI}:=\langle \tilde{F}^{(1)},\tilde{F}^{(2)},\tilde{F}^{(3)},\tilde{F}^{(4)}\rangle \subset \tilde{\cR}$ the {\em difference ideal} generated by the left-hand sides of Eqs.~\eqref{J-Basis}.

The elements in ${\cI}$ vanish on solutions of the Stokes flow~\eqref{Stokes} and those in $\tilde{\cI}$ vanish on solutions of~\eqref{J-Basis}. We refer to an element in ${\cI}$ (respectively, in $\tilde{\cI}$) as to a {\em consequence} of Eqs.~\eqref{Stokes} (respectively, of Eqs.~\eqref{J-Basis}).

\begin{definition}{\em\cite{GR'10}}
We shall say that a difference equation $\tilde{F}=0$ {\em implies the differential equation} $F=0$ and write
$\tilde{F}\rhd F$ when the Taylor expansion about a grid point yields
\begin{equation}
\tilde{F}\xrightarrow[h\rightarrow 0]{} F\cdot h^k + O(h^{k+1}),\ k\in \Z_{\geq 0}\,. \label{def-imply}
\end{equation}
\end{definition}
It is clear that to approximate Eqs.~\eqref{InvSys}, the scheme~\eqref{J-Basis} must be pairwise consistent with the involutive differential form~\eqref{InvSys}. We call this sort of consistency {\em weak consistency}.

\begin{definition}{\em\cite{GR'10}}
 A difference polynomial set $\{\tilde{F}^{(1)},\tilde{F}^{(2)},\tilde{F}^{(3)},\tilde{F}^{(4)}\}$~ is {\em weakly consistent} or {\em w-consistent} with differential system~\eqref{InvSys} if
\begin{equation}
   (\,\forall\ 1\leq i\leq 4\,)\ [\,\tilde{F}^{(i)}\rhd F^{(i)}\,]\,.
\end{equation}
\label{w-consistency}
\end{definition}

The following definition establishes the consistency interrelation between the differential and difference ideals generated by Eqs.~\eqref{Stokes} and Eqs.~\eqref{J-Basis}, respectively. If such a consistency holds, then it provides a certain inheritance of algebraic properties of Stokes flow by the difference scheme.

\begin{definition}{\em\cite{G'12}}
A finite difference approximation $\tilde{F}:=\{\tilde{F}^{(1)},\ldots,\tilde{F}^{(m)}\}$ to~\eqref{Stokes} is {\em strongly consistent} or {\em s-consistent} with Stokes flow~\eqref{Stokes} if
\begin{equation}
(\forall \tilde{f}\in \llbracket \tilde{F}\rrbracket)\  (\exists
f \in {\cI}) \ \ [\tilde{f}\rhd f]\,, \label{s-cond}
\end{equation}
where $\llbracket \tilde{F}\rrbracket$ is a {\em perfect difference ideal}~\cite{Levin'08} generated by the elements in the difference approximation.
\label{def-3}
\end{definition}

\begin{theorem}{\em\cite{G'12}}
The s-consistency condition~\eqref{s-cond} holds if and only if  a \Gr basis $\tilde{G}$ of $\tilde{\cI}$ satisfies
\begin{equation}
(\,\forall \tilde{g}\in \tilde{G}\,)\ (\,\exists f\in \langle F \rangle\,)\  [\,\tilde{g}\rhd f\,]\,. \label{s-consistency}
\end{equation}
\label{ThmSC}
\end{theorem}

\begin{corollary}\label{corollary}
The difference scheme~\eqref{J-Basis} is s-consistent with the Stokes system~\eqref{Stokes}.
\end{corollary}

\begin{proof}
By its construction, the set of difference polynomials in Eqs.~\eqref{J-Basis} is a Janet/\Gr basis of the elimination ideal $\tilde{\cI}_0\cap {\cR}$ where $\tilde{\cI}_0$ is the difference ideal generated by the polynomials in Eqs.~\eqref{DS} (cf.~\cite{AdamsLoustaunau'94}, Thm.~2.3.4). The same  set is also a Janet/\Gr basis for the ideal $\langle \tilde{F}^{(1)},\tilde{F}^{(2)},\tilde{F}^{(3)}\rangle$ and for the same POT ranking with $j\succ k$ and $u\succ v\succ p\succ f^{(1)}\succ f^{(2)}$. It is readily verified with the {\sc LDA} package. Furthermore, it is easy to see that
\begin{equation}\label{implication}
   \tilde{F}^{(i)} \rhd {F}^{(i)},\quad (i=1\div4)
\end{equation}
where ${F}^{(i)}$ are differential polynomials in Eqs.~\eqref{InvSys}. \hfill{$\Box$}
\end{proof}

\begin{remark}
For the computation of the image in mapping~\eqref{implication} one can use the command {\em ContinuousLimit} of the package {\sc LDA}.
\end{remark}

\noindent
It is clear that s-consistency of Eqs.\,\eqref{J-Basis} with~Eqs.~\eqref{Stokes} implies w-consistency. But the converse is not true. For the numerical simulation of the Stokes flow it is tempting to replace $\tilde{F}^{(4)}$ in Eqs.~\eqref{J-Basis} with a more compact discretization
\begin{equation}\label{dpe}
 \tilde{F}^{(4)}_1:= \Delta_1\left(p_{j,k}\right)
- \dfrac{f^{(1)}_{j+2,\,k+1} - f^{(1)}_{j,\,k+1}}{2h} - \dfrac{f^{(2)}_{j+1,\,k+2}-f^{(2)}_{j+1,\,k}}{2h}=0\,.
\end{equation}
Although this substitution preserves w-consistency since
\begin{equation}
 \tilde{F}^{(4)}_1\rhd F^{(4)}\,,
\end{equation}
the scheme $\{\,\tilde{F}^{(1)},\tilde{F}^{(2)},\tilde{F}^{(3)},\tilde{F}^{(4)}_1\,\}$ is not s-consistent.
\begin{proposition}\label{proposition}
 The difference scheme
 $\{\tilde{F}^{(1)},\tilde{F}^{(2)},\tilde{F}^{(3)},\tilde{F}^{(4)}_1\}$
 is s-inconsistent.
\end{proposition}

\begin{proof}
The difference polynomial~\eqref{dpe} does not belong to the difference ideal $\tilde{\cI}$ generated by the polynomial set in Eqs.~\eqref{J-Basis} since $\tilde{F}^{(4)}_1$ is irreducible modulo the ideal $\tilde{\cI}$. This can be shown by the direct computation of the normal form of $\tilde{F}^{(4)}_1$
modulo the Janet basis~\eqref{J-Basis} with the routine {\em InvReduce} of the Maple package {\sc LDA}. \hfill{$\Box$}
\end{proof}

Now let us analyse the set $\{\tilde{F}^{(1)},\tilde{F}^{(2)},\tilde{F}^{(3)},\tilde{F}^{(4)}_1\}$ with respect to  s-consistency. The Janet/\Gr basis of the difference ideal $\tilde{\cI}_1:=\langle \tilde{F}^{(1)},\tilde{F}^{(2)},\tilde{F}^{(3)},\tilde{F}^{(4)}_1 \rangle$ computed with {\sc LDA} shows that $\tilde{\cI}\neq \tilde{\cI}_1$. This basis consists of seven elements. Four of them, $\tilde{F}^{(1)},\tilde{F}^{(2)},\tilde{F}^{(3)},\tilde{F}^{(4)}_1$, imply system~\eqref{InvSys} and the three remaining elements, denoted by $\tilde{F}^{(5)}$, $\tilde{F}^{(6)}$ and $\tilde{F}^{(7)}$, are rather cumbersome difference equations which imply, respectively, the following differential ones
\begin{equation}\label{incon-eq}
\left\lbrace
\begin{array}{rl}
F^{(5)}:=& f^{(1)}_{xxxxx}+f^{(1)}_{xyyyy}+f^{(2)}_{xxxxy}+f^{(2)}_{yyyyy}=0\,,\\[4pt]
F^{(6)}:=& f^{(1)}_{xxx}-f^{(1)}_{xyy}+f^{(2)}_{xxy}-f^{(2)}_{yyy}+2\,p_{yyyy}=0\,,\\[4pt]
\end{array}
\right.
\end{equation}
and $\tilde{F}^{(7)}\rhd F^{(6)}$.

Equations~\eqref{incon-eq} are not consequences of the Stokes equations since the
differential polynomials $F^{(5)}$ and $F^{(6)}$ are irreducible modulo the differential ideal generated by the differential polynomials in Eqs.~\eqref{Stokes}. It follows that there are solutions to the Stokes equations which do not satisfy Eqs.~\eqref{incon-eq}.

\begin{remark}
Equations~\eqref{incon-eq} impose the limitations on the external forces which do not follow from the governing differential equations~\eqref{Stokes}. This is a result of s-inconsistency.
\end{remark}

\section{Modified Stokes Flow}
\label{ModifiedEquation}
In the framework of the method of {\em modified equation}~(cf.~\cite{Moin'10}, Sect.~5.5), a numerical solution of the governing differential system~\eqref{Stokes}, for given external forces $f^{(1)}$ and $f^{(2)}$, should be considered as a set of continuous differentiable functions $\{u,v,p\}$ whose values at the grid points satisfy the  difference scheme~\eqref{J-Basis}. Since the difference equations~\eqref{J-Basis} describe the differential ones~\eqref{InvSys} only approximately, we cannot expect that a continuous solution interpolating the grid values exactly satisfies Eqs.~\eqref{InvSys}. In reality, it satisfies another set of differential equations which we shall call the {\em modified steady Stokes flow} or {\em modified flow} for short.

Generally, the method of modified differential equation uses the representation of difference equations comprising the scheme as infinite order differential equations obtained by replacing the various shift operators in the difference equations by the Taylor series about a grid point. For equations of evolutionary type, the next step is to eliminate all derivatives with respect to the evolutionary variable of order greater than one. This step is done to obtain a kind of canonical form of the modified equation. Then, truncation of the order of the differential representations in the grid steps gives various modified equations (``differential approximations'') of the difference scheme.

As we show, the fact that both equation systems are \Gr bases of the ideals they generate and satisfy the condition~\eqref{implication} of s-consistency allows to develop a constructive procedure for the computation of the modified flow. Since the finite differences in the scheme~\eqref{J-Basis} approximate the partial derivatives occurring in Eqs.~\eqref{InvSys} with accuracy  $\mathcal{O}(h^2)$, it would appear reasonable that the scheme would have the second order of accuracy. For this reason, we restrict ourselves to the computation of the second order modified flow.

The Taylor expansions of the difference polynomials in Eqs.~\eqref{J-Basis} at the grid point $(-h,-h)$ for $\tilde{F}^{(1)},\tilde{F}^{(2)},\tilde{F}^{(3)},\tilde{F}^{(4)}_1$, and at the point $(-2h,-2h)$ for $\tilde{F}^{(4)}$ read
\begin{equation}\label{TaylorExp}
\left\lbrace
\begin{array}{rl}
\tilde{F}^{(1)}:=&  u_{x} + v_{y}+\frac{h^{2} u_{xxx}}{6} + \frac{h^{2} v_{yyy}}{6} + \mathcal{O}(h^4)=0\,,\\[6pt]
\tilde{F}^{(2)}:=&  p_{x} - \frac{1}{\mathrm{Re}}u_{xx} - \frac{1}{\mathrm{Re}}u_{yy}- f^{(1)}+\frac{h^{2} p_{xxx}}{6} - \frac{h^{2} u_{xxxx}}{12\,\mathrm{Re}} \\[6pt]
&  - \frac{h^{2} u_{yyyy}}{12\,\mathrm{Re}} +\mathcal{O}(h^4)=0\,,\\[6pt]
\tilde{F}^{(3)}:=& p_{y} - \frac{1}{\mathrm{Re}}v_{xx} - \frac{1}{\mathrm{Re}}v_{yy}- {f^{(2)}} +\frac{h^{2} p_{yyy}}{6} - \frac{h^{2} v_{xxxx}}{12\,\mathrm{Re}} \\[6pt]
& - \frac{h^{2} v_{yyyy}}{12\,\mathrm{Re}} + \mathcal{O}(h^4)=0\,,\\[6pt]
\tilde{F}^{(4)}:=&  p_{xx} + p_{yy} -  f^{(1)}_{x} - f^{(2)}_{y} - \frac{h^{2}f^{(1)}_{xxx}}{6} - \frac{h^{2}f^{(2)}_{yyy}}{6}\\[6pt]
& + \frac{h^{2} p_{xxxx}}{3} + \frac{h^{2} p_{yyyy}}{3} + \mathcal{O}(h^4)=0
\,,
\end{array}
\right.
\end{equation}
where the terms of order $h^2$ are written explicitly. The calculation of the right-hand sides in Eq.\eqref{TaylorExp} as well as the computation of the expressions given below was done with the use of freely available Python library {\sc SymPy} ({\tt http://www.sympy.org/}) for symbolic mathematics.

\begin{remark}\label{rem-evenh}
The Taylor expansions of the s-consistent difference scheme~\eqref{J-Basis} and of the s-inconsistent scheme $\{\tilde{F}^{(1)},\tilde{F}^{(2)},\tilde{F}^{(3)},\tilde{F}^{(4)}_1\}$ over the chosen grid points contain only the even powers of $h$. It follows immediately from the fact that all the finite differences occurring in the equations of both schemes are the central difference approximations of the partial derivatives occurring in~\eqref{InvSys}.
\end{remark}
Furthermore, we reduce the terms of order $h^2$ in the right-hand sides of~\eqref{TaylorExp} modulo the differential Janet/\Gr basis~\eqref{J-Basis}. This reduction will give us a {\em canonical form} of the second order modified flow, since given a \Gr basis, the {\em normal form} of a polynomial modulo this basis is uniquely defined~(cf.~\cite{AdamsLoustaunau'94}, Sect.~2.1). The normal form can be computed with the command {\em InvReduce} using the Maple package {\sc Janet}.

Thus, the Taylor expansion of the difference polynomials yields the {\em second order modified Stokes flow} as follows:
\begin{equation}\label{TaylorExpRed}
\left\lbrace
\begin{array}{rl}
\tilde{F}^{(1)}:=&  u_{x} + v_{y}+\frac{h^{2}\,\mathrm{Re}f^{(2)}_{y}}{6}
 - \frac{h^{2}\,\mathrm{Re}\, p_{yy}}{6} + \frac{h^{2} v_{yyy}}{3} + \mathcal{O}(h^4)=0\,,\\[6pt]
\tilde{F}^{(2)}:=&  p_{x} + \frac{1}{\mathrm{Re}}v_{xy} - \frac{1}{\mathrm{Re}}u_{yy}- f^{(1)}+\frac{h^{2}f^{(1)}_{xx}}{6} + \frac{h^{2}f^{(1)}_{yy}}{4} + \frac{h^{2}f^{(2)}_{xy}}{4} \\[6pt]
& - \frac{h^{2} p_{xyy}}{2} + \frac{h^{2} u_{yyyy}}{6\,\mathrm{Re}} +\mathcal{O}(h^4)=0\,,\\[6pt]
\tilde{F}^{(3)}:=& p_{y} - \frac{1}{\mathrm{Re}}v_{xx} - \frac{1}{\mathrm{Re}}v_{yy}- {f^{(2)}} - \frac{h^{2}f^{(1)}_{xy}}{12} + \frac{h^{2} f^{(2)}_{xx} }{12} - \frac{h^{2}f^{(2)}_{yy}}{6} \\[6pt]
& + \frac{h^{2} p_{yyy}}{3} - \frac{h^{2} v_{yyyy}}{6\,\mathrm{Re}}  + \mathcal{O}(h^4)=0\,,\\[6pt]
\tilde{F}^{(4)}:=&  p_{xx} + p_{yy} -  f^{(1)}_{x} - f^{(2)}_{y} + \frac{h^{2}f^{(1)}_{xxx}}{6} - \frac{h^{2} f^{(1)}_{xyy}}{3}  \\[6pt]
& + \frac{h^{2}f^{(2)}_{xxy}}{3} - \frac{h^{2} f^{(2)}_{yyy}}{2}  + \frac{2h^{2} p_{yyyy}}{3}  + \mathcal{O}(h^4)=0
\,.
\end{array}
\right.
\end{equation}

\begin{remark}
Note that the symmetry under the swap transformation~\eqref{SymTrans} that holds in Eqs.~\eqref{TaylorExp} does not hold in Eqs.~\eqref{TaylorExpRed}. This symmetry breaking is a typical effect of the application of the \Gr reduction to symmetric systems and caused by the non-symmetry of the term ordering.
\end{remark}

As we know, Stokes flow~\eqref{Stokes} satisfies the integrability condition~\eqref{IntCon} which we rewrite as
\begin{equation}\label{IntRel}
F^{(2)}_x + F^{(3)}_y + \frac{1}{\mathrm{Re}}\left( F^{(1)}_{xx} + F^{(1)}_{yy}\right) - F^{(4)} = 0\,.
\end{equation}

Substitution of the Taylor expansions~\eqref{TaylorExpRed} into the equality~\eqref{IntRel} shows that the sum of the second-order terms explicitly written in formulae~\eqref{TaylorExpRed} is equal to zero.
The following proposition shows that this is a consequence of the s-consistency of the scheme.

\begin{proposition}\label{S-cons}
Given a uniform and orthogonal solution grid with a spacing $h$, a w-consistent difference scheme for Eqs.~\eqref{InvSys} is s-consistent only if its Taylor expansion based on the central-difference formulas for derivatives and reduced modulo system~\eqref{InvSys}, after its substitution into the left-hand side of the equality~\eqref{IntRel}, vanishes for  every order in $h^2$.
\end{proposition}

\begin{proof} Let $\tilde{G}:=\{\tilde{G}^{(1)},\tilde{G}^{(2)},\tilde{G}^{(3)},\tilde{G}^{(4)}\}$ be a set of s-consistent difference approximations to the differential polynomials~$F^{(1)},F^{(2)},F^{(3)},F^{(4)}$ in the Janet/\Gr basis~\eqref{InvSys}. The w-consistency of $\tilde{G}$ implies the central difference Taylor expansion
\begin{equation}\label{ExpG}
   \tilde{G}^{(i)}=F^{(i)}+\sum_{m=1}^\infty h^{2m}r^{(i)}_m\,,\quad r^{(i)}_m\in \Q(\mathrm{Re})[u,v,p,f^{(1)},f^{(2)}]\quad (i=1\div4)\,.
\end{equation}
We consider the family of difference polynomials $(m\in \N_{\geq 1})$
\begin{equation}\label{accuraceIC}
\tilde{G}_0^{(m)}:=D_1^{(m)}\tilde{G}^{(2)} + D_2^{(m)}\tilde{G}^{(3)} + \frac{1}{\mathrm{Re}}\left( D_{1,1}^{(m)}\tilde{G}^{(1)} + D_{2,2}^{(m)}\tilde{G}^{(1)}\right) - \tilde{G}^{(4)}
\end{equation}
with the central-difference operators $D_1^{(m)}$, $D_2^{(m)}$, $D_{1,1}^{(m)}$, $D_{2,2}^{(m)}$ approximating the partial differential operators $\partial_x$,  $\partial_y$, $\partial_{xx}$, $\partial_{yy}$ with accuracy $h^{2m}$. Apparently, $\tilde{G}_0^{(m)}$  belongs to the perfect difference ideal generated by $\tilde{G}$:
\[
\left(\forall m\in \N_{\geq 1}\right)\ \  [\,\tilde{G}_0^{(m)}\in \llbracket \tilde{G}\rrbracket\,]\,.
\]
These difference operators are composed of the translations~\eqref{shifts}. For example,
\[
       D_{i}^{(1)}:=\frac{\sigma_{i}-\sigma_{i}^{-1}}{2h}\,,\quad D_{i,i}^{(1)}:=\frac{\sigma_{i}-2+\sigma_{i}^{-1}}{h^2}\,,\quad i\in \{1,2\}
\]
and
\[
     D_{i}^{(2)}:=\frac{-\sigma_{i}^2+8\sigma_{i}-8\sigma_{i}^{-1}+\sigma_{i}^{-2}}{12h}\,,\quad D_{i,i}^{(2)} :=\frac{-\sigma_{i}^2+16\sigma_{i}-30+16\sigma_{i}^{-1}-\sigma_{i}^{-2}}{12h^2}
\]
with $\sigma_{1}^{-1}\circ u(j,\,k)=u(j-1,\,k),\,\sigma_2^{-1}\circ u(j,\,k)=u(j,\,k-1)$, etc., $\sigma_i^2=\sigma_i\circ \sigma_i$ and $\sigma_{i}^{-2}=\sigma_{i}^{-1}\circ \sigma_{i}^{-1}$.

From Eqs.~\eqref{ExpG} and \eqref{accuraceIC}, we obtain
\begin{eqnarray}
\tilde{G}^{(1)}_0 &=& F^{(2)}_x + F^{(3)}_y + \frac{1}{\mathrm{Re}}\left( F^{(1)}_{xx} + F^{(1)}_{yy}\right) - F^{(4)} + \mathcal{O}(h^2)\,,\nonumber \\
& \Rightarrow & F^{(2)}_x + F^{(3)}_y + \frac{1}{\mathrm{Re}}\left( F^{(1)}_{xx} + F^{(1)}_{yy}\right) - F^{(4)}=0\,, \label{G1} \\
\tilde{G}^{(2)}_0 &=& h^2\left(\partial_x r^{(2)}_1+\partial_y r^{(3)}_1+\frac{1}{\mathrm{Re}}\left( \partial_{xx}r^{(1)}_1 + \partial_{yy}r^{(1)}_1\right)\right)+\mathcal{O}(h^4) \nonumber\\
& \Rightarrow & \partial_x r^{(2)}_1+\partial_y r^{(3)}_1+\frac{1}{\mathrm{Re}}\left( \partial_{xx}r^{(1)}_1 + \partial_{yy}r^{(1)}_1\right)=0 \,,  \label{G2}\\
&\vdots&\nonumber\\
\tilde{G}^{(k)}_0 &=& h^{2k}\left(\partial_x r^{(1)}_k+\partial_y r^{(3)}_k+\frac{1}{\mathrm{Re}}\left( \partial_{xx}r^{(1)}_k + \partial_{yy}r^{(2)}_k\right)\right)+\mathcal{O}(h^{2k+2}) \nonumber\\
& \Rightarrow & \partial_x r^{(2)}_k+\partial_y r^{(3)}_k+\frac{1}{\mathrm{Re}}\left( \partial_{xx}r^{(1)}_k + \partial_{yy}r^{(1)}_k\right)=0 \,,  \dots\,.\label{GK}
\end{eqnarray}

The implication in Eq.~\eqref{G2} follows from the fact that the normal form of the differential polynomial~\eqref{G2} modulo Eqs.~\eqref{InvSys}, if it is nonzero, does not belong to the differential ideal generated by the polynomials in \eqref{InvSys} that contradicts the s-consistency of $\tilde{G}$. Because of the same argument, the equality~\eqref{GK} holds for any $k$.  \hfill{$\Box$}
\end{proof}

\begin{corollary}\label{CorSC}
A w-consistent difference scheme for system~\eqref{InvSys} is s-consistent if and only if its set of polynomials is a difference Janet/\Gr basis for the POT ranking~\eqref{ranking}.
\end{corollary}

\begin{proof}
``$\Leftarrow$'' Because of our choice~\eqref{ranking} of the ranking and the structure~\eqref{InvSys}, differential Janet/\Gr basis with the underlined leaders, a w-consistent difference scheme composed of four  difference polynomials $\{\tilde{G}^{(1)},\tilde{G}^{(2)},\tilde{G}^{(3)},\tilde{G}^{(4)}\}$ has the only difference $S$-polynomial of the form~\eqref{accuraceIC} which approximates the left-hand side of the differential integrability condition~\eqref{IntRel}. Together with the Taylor expansion~\eqref{ExpG}, the relations~\eqref{G1}--\eqref{GK} imply the reduction of S-polynomial~\eqref{accuraceIC} to zero modulo $\{\tilde{G}^{(1)},\tilde{G}^{(2)},\tilde{G}^{(3)},\tilde{G}^{(4)}\}$. Thus, the scheme is a Janet/\Gr basis.

``$\Rightarrow$'' If a w-consistent set $\{\tilde{G}^{(1)},\tilde{G}^{(2)},\tilde{G}^{(3)},\tilde{G}^{(4)}\}$ is a Janet/\Gr basis, then by Theorem~\ref{ThmSC} it is s-consistent.  \hfill{$\Box$}
\end{proof}

We illustrate Proposition~\ref{S-cons} and Corollary~\ref{CorSC} by the s-inconsistent difference scheme $\{\tilde{F}^{(1)}_1,\tilde{F}^{(2)}_1,\tilde{F}^{(3)}_1,\tilde{F}^{(4)}_1\}$ of Section~\ref{sec:consistency} where  the first three difference equations coincide with those of the system~\eqref{J-Basis},
\[
  \tilde{F}^{(i)}_1 =\tilde{F}^{(i)} \quad (i=1,2,3)\,,
\]
and $\tilde{F}^{(4)}_1$ is given by Eq.~\eqref{dpe}. Because of the distinction of the last equation from $\tilde{F}^{(4)}$ in~\eqref{J-Basis}, the reduced Taylor expansions of equations $\tilde{F}^{(1)}_1=0$ and  $\tilde{F}^{(4)}_1=0$ are different from $\tilde{F}^{(1)}=0$ and $\tilde{F}^{(4)}=0$ in system~\eqref{TaylorExpRed}:
\begin{equation}\label{TaylorExpRed1}
\left\lbrace
\begin{array}{rl}
\tilde{F}^{(1)}_1:=&  u_{x} + v_{y}
 - \frac{h^{2}\,\mathrm{Re}\, p_{yy}}{6} + \frac{h^{2} v_{yyy}}{3} + \mathcal{O}(h^4)=0\,,\\[6pt]
\tilde{F}^{(2)}_1:=&  p_{x} + \frac{1}{\mathrm{Re}}v_{xy} - \frac{1}{\mathrm{Re}}u_{yy}- f^{(1)}+\frac{h^{2}f^{(1)}_{xx}}{6} + \frac{h^{2}f^{(1)}_{yy}}{4} + \frac{h^{2}f^{(2)}_{xy}}{4} \\[6pt]
& - \frac{h^{2} p_{xyy}}{2} + \frac{h^{2} u_{yyyy}}{6\,\mathrm{Re}} +\mathcal{O}(h^4)=0\,,\\[6pt]
\tilde{F}^{(3)}_1:=& p_{y} - \frac{1}{\mathrm{Re}}v_{xx} - \frac{1}{\mathrm{Re}}v_{yy}- {f^{(2)}} - \frac{h^{2}f^{(1)}_{xy}}{12} + \frac{h^{2} f^{(2)}_{xx} }{12} - \frac{h^{2}f^{(2)}_{yy}}{6} \\[6pt]
& + \frac{h^{2} p_{yyy}}{3} - \frac{h^{2} v_{yyyy}}{6\,\mathrm{Re}}  + \mathcal{O}(h^4)=0\,,\\[6pt]
\tilde{F}^{(4)}_1:=&  p_{xx} + p_{yy} -  f^{(1)}_{x} - f^{(2)}_{y} - \frac{h^{2}f^{(1)}_{xxx}}{12} - \frac{h^{2} f^{(1)}_{xyy}}{12}  \\[6pt]
& + \frac{h^{2}f^{(2)}_{xxy}}{12} - \frac{h^{2} f^{(2)}_{yyy}}{4}  + \frac{h^{2} p_{yyyy}}{6}  + \mathcal{O}(h^4)=0
\,.
\end{array}
\right.
\end{equation}
If we expand $\tilde{F}^i_1$ $(i=1\div4)$ up to the fourth order terms in $h$ and substitute the obtained expansions into the left-hand side of the integrability condition~\eqref{IntRel}, then we obtain
\begin{equation}\label{incons-ic}
\frac{h^{2}f^{(1)}_{xxx}}{4} - \frac{h^{2}f^{(1)}_{xyy}}{4} + \frac{h^{2}f^{(2)}_{xxy}}{4} - \frac{h^{2}f^{(2)}_{yyy}}{4} + \frac{h^{2} p_{{yyyy}}}{2}+\mathcal{O}(h^4).
\end{equation}

Expression~\eqref{incons-ic} contains terms of second order in $h$. Up to the factor 4, the sum of these terms is the differential polynomial $F^{(6)}$ in Eqs.~\eqref{incon-eq}. Thus, the presence of the second-order terms in \eqref{incons-ic} is intimately related to the s-inconsistency of~\eqref{TaylorExpRed} with governing Stokes equations~\eqref{Stokes}. It is clear that the PDE system~\eqref{TaylorExpRed1} cannot be considered as a modified Stokes flow.

\section{Numerical Simulation}
\label{sec:Numerics}

In this section, we present a numerical simulation in order to experimentally validate the s-consistent difference scheme~\eqref{J-Basis} for which we constructed the modified Stokes flow~\eqref{TaylorExpRed}. For that, we suppose that the Stokes system (\ref{Stokes}) is defined in the rectangular domain which is discretized in the $x$- and $y$-directions by means of equidistant points. We simulate a fluid flow through porous media which is often mainly caused by the viscous forces, so that its modeling using the Stokes system (\ref{Stokes}) is reasonable; see Fig.~\ref{PorousMediaFlow}. Such a setup has many practical applications in the field of petroleum engineering \cite{Fancher'33}.

\begin{figure}[h]
    \centering
    \includegraphics[width=\textwidth]{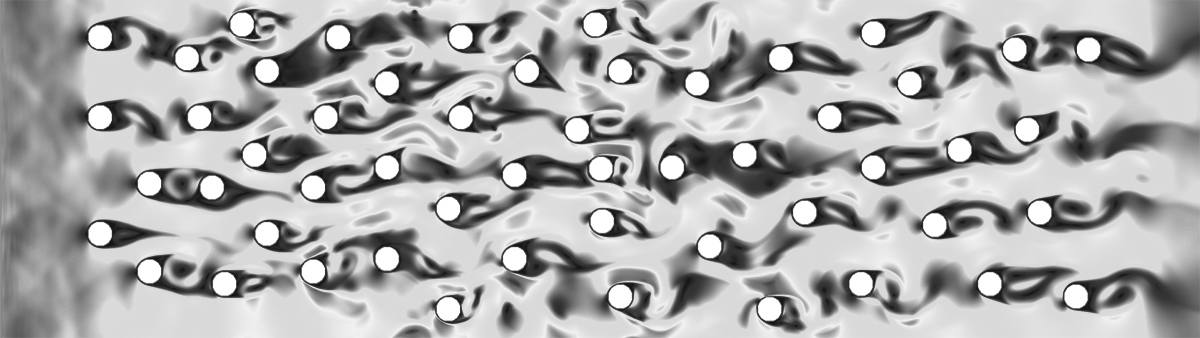}
    \caption{Visualization of  the simulation of a fluid flow through porous media using the s-consistent difference scheme~\eqref{J-Basis}.}
    \label{PorousMediaFlow}
\end{figure}

We measure the maximum relative error of the average velocities compared to a ground truth result obtained by computing with extremely tiny $h$-values. From several simulations with varying $h$-values, we can follow that a maximum relative error of more than 15\% in the velocity space compared to the ground truth should not be tolerated in order to ensure for a sufficient degree of global accuracy. Using this restriction we evaluate the performance of the s-consistent difference scheme~\eqref{J-Basis} against the popular classic marker-and-cell (MAC) method \cite{Harlow'65}. We observe that compared to MAC, using the scheme~\eqref{J-Basis}, one can simulate with around a factor of $1.7$, i.e.,~with significantly larger $h$-values and, at the same time, keep the relative error below the 15\%-bar. Moreover, we observe that this factor is only slightly dependent on the Reynolds number.

\section{Conclusion}
\label{Conclusion}
For the two-dimensional incompressible steady Stokes flow~\eqref{Stokes} and a regular Cartesian solution grid, we  presented a computer algebra-based approach in order to derive the s-consistent difference scheme~\eqref{J-Basis} for which we constructed the modified Stokes flow~\eqref{TaylorExpRed}. It shows that the generated scheme has order $\mathcal{O}(h^2)$.

Our computational procedure for the derivation of the modified Stokes flow is based on a combination of differential and difference \Gr basis techniques. The first is applied to the governing Stokes equations~\eqref{Stokes} to complete them to the involution form~\eqref{InvSys}  incorporating the pressure Poisson equation $F^{(4)}$, and to verify the s-consistency of the scheme by applying the criterion of s-consistency (Theorem~\ref{ThmSC}) which is fully algorithmic for linear systems of PDEs. The difference \Gr bases technique is used for the derivation of the scheme on the chosen grid by means of difference elimination.

In addition, we used both techniques to construct a modified Stokes flow~\eqref{TaylorExpRed}. Its structure as well as that of the scheme depends on the used difference ranking. We experimented with several rankings and finally preferred the POT ranking satisfying~\eqref{ranking1} for the differential case and~\eqref{ranking} for the difference case as the best suited. To perform the related computations we used the Maple packages {\sc Janet}~\cite{Maple-Janet'03} and {\sc LDA}~\cite{GR'12}.

Since our difference scheme~\eqref{J-Basis} for ranking~\eqref{ranking} is obtained from its first three equations $\{\tilde{F}^{(1)},\tilde{F}^{(2)},\tilde{F}^{(3)}\}$ by constructing the difference Janet/\Gr basis (see Remark~\ref{Rem2}), it is interesting to check via the \Gr bases whether there are approximations of the continuity equation $F^{(1)}$ in the difference ideal generated by $\tilde{F}:=\{\tilde{F}^{(2)},\tilde{F}^{(3)},\tilde{F}^{(4)}\}$. In the case of existence of such approximations they might be used for the numerical study of Stokes flow in the velocity-pressure formulation. However, the computation with {\sc LDA} shows that the discrete version of $F^{(1)}$ is not a consequence of $\tilde{F}$. Thus, in the velocity-pressure formulation one has to add information on the continuity equation to $\tilde{F}$ via the corresponding boundary condition (cf.~\cite{Petersson'01}).

\section*{Acknowledgements}

The authors are grateful to Daniel Robertz for his help with respect to the use of the packages {\sc Janet} and {\sc LDA} and to the anonymous referees for their suggestions. This work has been partially supported by the King Abdullah University of Science and Technology (KAUST baseline funding), the Russian Foundation for Basic Research (16-01-00080) and the RUDN University Program (5-100).



\end{document}